\documentclass{amsart}

\usepackage{amsmath}
\usepackage{amsthm}
\usepackage{color}
\usepackage{amsfonts}
\usepackage{amssymb}
\usepackage{mathrsfs}
\usepackage{amscd}
\usepackage{url}

\newcommand{\re}{\mathbb{R}}

\newcommand{\N}{\mathbb{N}}

\newcommand{\half}{\frac{1}{2}}
\newcommand{\lmd}{\lambda}

\newcommand{\eps}{\epsilon}

\newcommand{\dt}{\delta}

\def\af{\alpha}
\def\bt{\beta}

\def\rank{\mbox{rank}}

\newcommand{\sig}{\sigma}
\newcommand{\Sig}{\Sigma}

\newcommand{\reff}[1]{(\ref{#1})}

\newcommand{\mc}[1]{\mathcal{#1}}

\newcommand{\bdes}{\begin{description}}
\newcommand{\edes}{\end{description}}

\newcommand{\bal}{\begin{align}}
\newcommand{\eal}{\end{align}}

\newcommand{\bnum}{\begin{enumerate}}
\newcommand{\enum}{\end{enumerate}}

\newcommand{\bit}{\begin{itemize}}
\newcommand{\eit}{\end{itemize}}

\newcommand{\bea}{\begin{eqnarray}}
\newcommand{\eea}{\end{eqnarray}}
\newcommand{\be}{\begin{equation}}
\newcommand{\ee}{\end{equation}}

\newcommand{\baray}{\begin{array}}
\newcommand{\earay}{\end{array}}

\newcommand{\bsry}{\begin{subarray}}
\newcommand{\esry}{\end{subarray}}

\newcommand{\bca}{\begin{cases}}
\newcommand{\eca}{\end{cases}}

\newcommand{\bcen}{\begin{center}}
\newcommand{\ecen}{\end{center}}

\newcommand{\bbm}{\begin{bmatrix}}
\newcommand{\ebm}{\end{bmatrix}}

\newcommand{\bmx}{\begin{matrix}}
\newcommand{\emx}{\end{matrix}}

\newcommand{\bpm}{\begin{pmatrix}}
\newcommand{\epm}{\end{pmatrix}}

\newcommand{\btab}{\begin{tabular}}
\newcommand{\etab}{\end{tabular}}

\newtheorem{theorem}{Theorem}[section]
\newtheorem{pro}[theorem]{Proposition}

\newtheorem{lem}[theorem]{Lemma}

\newtheorem{cor}[theorem]{Corollary}

\newtheorem{ass}[theorem]{Assumption}

\theoremstyle{definition}

\newtheorem{exm}[theorem]{Example}

\newtheorem{remark}[theorem]{Remark}

\begin{document}

\title[Flat Truncation and Lasserre's Hierarchy]
{Certifying Convergence of Lasserre's Hierarchy via Flat Truncation}

\author[Jiawang Nie]{Jiawang Nie}
\address{Department of Mathematics\\
  University of California \\
  San Diego}
\email{njw@math.ucsd.edu}

\begin{abstract}
Consider the optimization problem of minimizing a polynomial function
subject to polynomial constraints.
A typical approach for solving it globally is applying Lasserre's hierarchy of
semidefinite relaxations, based on either
Putinar's or Schm\"{u}dgen's Positivstellensatz.
A practical question in applications is: how to certify its
convergence and get minimizers?
In this paper, we propose {\it flat truncation} as a certificate
for this purpose. Assume the set of global minimizers
is nonempty and finite. Our main results are:
i) Putinar type Lasserre's hierarchy has finite convergence
if and only if flat truncation holds, under some generic assumptions;
the same conclusion holds for the Schm\"{u}dgen type one
under weaker assumptions.
ii) Flat truncation is asymptotically satisfied
for Putinar type Lasserre's hierarchy if the archimedean condition holds;
the same conclusion holds for the Schm\"{u}dgen type one
if the feasible set is compact.
iii) We show that flat truncation can be used as a certificate
to check exactness of standard SOS relaxations and Jacobian SDP relaxations.
\end{abstract}

\keywords{flat truncation, Lasserre's relaxations, quadratic module, preordering,
semidefinite programming}

\subjclass{ 65K05, 90C22}

\maketitle

\section{Introduction}

Given polynomials $f,g_1,\ldots,g_m$, consider the optimization problem
\be \label{pop:K}
\left\{\baray{rl}
\underset{x\in \re^n}{\min} &  f(x)\\
\mbox{s.t.}& g_1(x)\geq 0, \ldots, g_m(x) \geq 0.
\earay \right.
\ee
Let $K$ be its feasible set, and $f_{min}$ be its global minimum value.
The $k$-th Lasserre's relaxation \cite{Las01} for solving \reff{pop:K} is
($k$ is also called the {\it relaxation order})
\be \label{las-rex:put}
\max \quad \gamma \quad
\mbox{ s.t. }  f - \gamma \in Q_k(g).
\ee
Here, the set $Q_k(g)$ denotes the $k$-th truncated
quadratic module generated by the tuple $g:=(g_1,\ldots,g_m)$
(for convenience, denote $g_0 = 1$
and say a polynomial is SOS if it is a sum of squares of
polynomials with real coefficients):
\[
Q_k(g) := \left\{ \left.\sum_{i=0}^m g_i \sig_i \right|
\sig_i \mbox{ is SOS}, \, \deg(g_i\sig_i) \leq 2k \mbox{ for every } i
\right\}.
\]
The relaxation~\reff{las-rex:put} is equivalent to a semidefinite program (SDP),
and thus could be solved efficiently by numerical methods
like interior point algorithms.
Let $f_k$ be the optimal value of \reff{las-rex:put} for a given order $k$.
Clearly, every $f_k  \leq f_{min}$
and the sequence $\{f_k\}$ is monotonically increasing.
Under the archimedean condition (i.e., there exists $\phi \in Q_\ell(g)$ for some $\ell$
such that the inequality $\phi(x)\geq 0$ defines a compact set in $x$),
Lasserre proved the convergence $f_k \to f_{min}$ as $k \to \infty$ \cite{Las01}.
An estimation of its convergence rate is given in \cite{NieSwg}.
The sequence of \reff{las-rex:put} as $k\to\infty$
is called {\it Lasserre's hierarchy} in the literature.
As demonstrated by extensive numerical experiments,
it occurs quite a lot that $f_k = f_{min}$ for a finite order $k$ in applications.
If this happens, we say Lasserre's hierarchy has {\it finite convergence}.
This raises a very practical question: since $f_{min}$ is typically unknown,
how do we certify its finite convergence if it happens?
If it is certified, how do we get minimizers?
A frequently used sufficient condition for this purpose is {\it flat extension}
introduced by Curto and Fialkow (cf. \cite{CF05}),
but it is a strong condition that might not be satisfied,
i.e., it is not necessary. To the author's best knowledge,
there is very little work on proving a certificate
for checking finite convergence of Lasserre's hierarchy,
and the question is almost completely open.
The motivation of this paper is to address this issue.
Our main result is that a more suitable condition called {\it flat truncation}
could generically serve as such a certificate.

\subsection{Background}

Typically, to extract a global minimizer, one needs to consider
the dual optimization problem of \reff{las-rex:put}
whose description uses localizing matrices. Define degree integers
\be \label{dfdid}
d_i = \lceil \deg(g_i)/2 \rceil, \quad
d_g = \max\{1, d_1,\ldots, d_m\}, \quad
d_f = \lceil \deg(f)/2 \rceil.
\ee
(Here $\lceil a \rceil$ denotes the smallest integer that
is greater than or equal to $a$.)
Let $y$ be a sequence indexed by $\af:=(\af_1,\ldots,\af_n)\in \N^n$
($\N$ is the set of nonnegative integers)
with $|\af| := \af_1+\cdots+\af_n \leq 2k$,
i.e., $y$ is a {\it truncated moment sequence (tms)} of degree $2k$.
Denote by $\mathscr{M}_{2k}$ the space of all tms whose degrees are $2k$.
A tms $y \in \mathscr{M}_{2k}$ defines a Riesz functional $\mathscr{L}_y$ on $\re[x]_{2k}$
(the space of real polynomials in $x:=(x_1,\ldots,x_n)$
with degrees at most $2k$, and denote $\re[x]:=\sum_i \re[x]_i$) as
\[
\mathscr{L}_y\left(\sum_\af
p_\af x_1^{\af_1}\cdots x_n^{\af_n} \right) := \sum_\af p_\af y_\af.
\]
For convenience, denote $x^\af := x_1^{\af_1}\cdots x_n^{\af_n}$ and
\[
\langle p, y \rangle := \mathscr{L}_y(p).
\]
The $k$-th {\it localizing matrix} $L_{h}^{(k)}(y)$
generated by a polynomial $h$ and a tms $y \in \mathscr{M}_{2k}$
is a symmetric matrix satisfying (denote $d_h := \lceil \deg(h)/2 \rceil$)
\[
p^T L_{h}^{(k)}(y) q  := \mathscr{L}_y(h pq)
\quad \forall\, p, q \in \re[x]_{k-d_h}.
\]
(Here, for convenience,  we still use $p$ to
denote the coefficient vector of the polynomial $p$.)
When $h = 1$, $L_{h}^{(k)}(y)$ is
called a {\it moment matrix} and is denoted as
\[
M_k(y):= L_{1}^{(k)}(y).
\]
The columns and rows of $L_{h}^{(k)}(y)$, as well as $M_k(y)$,
are indexed by integral vectors $\af \in \N^n$ with $|\af|\leq k-d_h$.

The dual optimization problem of \reff{las-rex:put} is (cf. \cite{Las01,LasBok})
\be \label{mom-las:put}
\left\{\baray{rl}
\underset{ y \in \mathscr{M}_{2k} }{\min} &  \langle f, y \rangle \\
\mbox{s.t.}& L_{g_i}^{(k)}(y) \succeq 0 \,(i=0,1,\ldots, m), \,
\langle 1, y \rangle = 1.
\earay \right.
\ee
In the above, $X\succeq 0$ means a matrix $X$ is positive semidefinite.
Let $f_k^*$ be the optimal value of \reff{mom-las:put} for order $k$.
By weak duality, $f_k^* \leq f_k$ for every $k$.
If $K$ has nonempty interior, then \reff{mom-las:put} has an interior point,
\reff{las-rex:put} achieves its optimal value
and $f_k^* = f_k$, i.e., there is no duality gap between
\reff{las-rex:put} and \reff{mom-las:put} (cf. \cite{Las01}).
Clearly, every $f_k^* \leq f_{min}$, and the sequence
$\{f_k^*\}$ is also monotonically increasing.
We refer to Lasserre's book \cite{LasBok} and
Laurent's survey \cite{Lau} for related work in this area.

Suppose $y^*$ is an optimizer of \reff{mom-las:put}.
If the {\it flat extension} condition (cf. \cite{CF05})
\be \label{con:fec}
\rank \, M_{k-d_g}(y^*) \quad = \quad \rank \, M_{k}(y^*)
\ee
holds ($d_g$ is from \reff{dfdid}),
then we can extract $r=\rank \,M_k(y^*)$ global optimizers for \reff{pop:K}.
By Theorem~1.1 of Curto and Fialkow \cite{CF05},
if $y^*$ is feasible for \reff{mom-las:put} and \reff{con:fec} is satisfied,
then $y^*$ admits a unique $r$-atomic measure supported on $K$,
i.e., there exist $(\lmd_1, \ldots, \lmd_r)>0$
and $r$ distinct points $v_1,\ldots, v_r \in K$ such that
\be \label{r-atom:y*}
y^* \quad = \quad  \lmd_1 [v_1]_{2k} + \cdots + \lmd_r [v_r]_{2k}.
\ee
Here, for $x\in \re^n$, $[x]_{2k}$ denotes the vector defined as
\[
[x]_{2k} := \bbm 1 & x_1 &\cdots & x_n & x_1^2 & x_1x_2 & \cdots & x_n^{2k}\ebm^T.
\]
In \reff{r-atom:y*}, the constraint $\langle 1, y^* \rangle = 1$ implies
$\lmd_1 + \cdots + \lmd_r =1$. If \reff{con:fec} holds, then
$f_k=f_{min}$ (suppose there is no duality gap
between \reff{las-rex:put} and \reff{mom-las:put}),
all $v_1,\ldots, v_r$ are global minimizers of \reff{pop:K},
and they could typically be obtained by solving some SVD and eigenvalue problems,
as shown by Henrion and Lasserre \cite{HenLas05}.
Generally, \reff{con:fec} is a sufficient but not necessary condition
for checking finite convergence of Lasserre's hierarchy.

For a tms $z\in \mathscr{M}_{2t}$, we say $z$ is {\it flat with respect to $g$} if
$z$ is feasible in \reff{mom-las:put} for $k=t$ and satisfies the condition
$\rank \, M_{t-d_g}(z) = \rank \, M_{t}(z)$ (cf. \cite{CF05}).
If the tuple $g$ is clear in the context, we just simply say $z$ is flat.

\subsection{Flat truncation}
\label{subsec:flat}

To get a minimizer of \reff{pop:K} from an optimizer $y^*$ of \reff{mom-las:put},
the flat extension condition \reff{con:fec} would be weakened.
For instance, if $y^*$ has a {\it flat truncation}, i.e.,
there exists an integer $ t\in [\max\{d_f, d_g\},k]$ such that
\be \label{cd:flat-trun}
\rank \, M_{t-d_g}(y^*) \quad = \quad \rank \, M_{t}(y^*),
\ee
then we could still extract $r=\rank M_{t}(y^*)$ minimizers for \reff{pop:K}.
Let $z=y^*|_{2t}$ (denote $y^*|_{2t}:=(y^*_\af)_{|\af| \leq 2t}$) be the truncation.
Then, the tms $z$ is flat, and there exist $r$ distinct points $u_1,\ldots, u_r \in K$
and scalars $\lmd_1,\ldots,\lmd_r$ such that
\[
z = \lmd_1 [u_1]_{2t} + \cdots + \lmd_r [u_r]_{2t},
\]
\[
\lmd_1 >0,\ldots,\lmd_r>0, \quad \lmd_1+\cdots+\lmd_r = 1.
\]
Clearly, $\langle f, z\rangle = \langle f, y^* \rangle = f_k^*
\leq f_{min}$ (note $2t \geq \deg(f)$) and
\[
\lmd_1 f(u_1) + \cdots + \lmd_r f(u_r) = \langle f, z\rangle =
\langle f, y^* \rangle   \leq f_{min}.
\]
Since every $f(u_i) \geq f_{min}$, each $u_i$ must be a global minimizer
of \reff{pop:K}. Hence,
if an optimizer of \reff{mom-las:put} has a flat truncation,
then $f_k^* = f_{min}$, and $f_k = f_{min}$ if there is no
duality gap between \reff{las-rex:put} and \reff{mom-las:put}.

The flat extension condition \reff{con:fec}
and flat truncation condition \reff{cd:flat-trun}
are different in checking convergence of Lasserre's hierarchy.
We show the difference in the following example.

\begin{exm} \label{emp:FTneFE}
Consider the univariate cubic optimization problem
\[
\min_{x\in \re} \quad x^3 \quad \mbox{s.t.} \quad x \geq 0, 1-x\geq 0.
\]
The feasible set is compact and has interior points,
and the archimedean condition holds. The global minimum $f_{min}=0$
and the origin is the unique global minimizer.
It could be easily seen that $f_k^*=f_k=0$ for all $k\geq 2$.
So, its Lasserre's hierarchy has finite convergence.
Both \reff{las-rex:put} and \reff{mom-las:put} achieve their optimal values.
However, the flat extension condition \reff{con:fec}
is violated for almost all optimizers of \reff{mom-las:put},
while the flat truncation condition \reff{cd:flat-trun}
is satisfied for all the optimizers.
To see this, let $y^*$ be an optimizer of \reff{mom-las:put},
then $y^*_3 = 0$. It could be easily shown that
the constraints of \reff{mom-las:put} implies
that $y^*_\af = 0$ for all $0<|\af|\leq 2k-1$.
(For instance, for $k=2$, its constraints are
\[
M_2(y) = \bbm  1 & y_1 & y_2  \\ y_1 & y_2  & y_3 \\
y_2 & y_3 & y_4   \ebm \succeq 0, \qquad
L_{x}^{(2)}(y) = \bbm  y_1 & y_2  \\ y_2  & y_3  \ebm \succeq 0,
\]
\[
L_{1-x}^{(3)}(y) = \bbm  1-y_1 & y_1-y_2  \\  y_1-y_2  & y_2-y_3  \ebm \succeq 0,
\]
and $y^*_3=0$ implies $y^*_1=y^*_2=0$.)
Thus, for each $k\geq 2$, every optimizer of \reff{mom-las:put} is in the form
\[
y^*(\eps):=(1, 0, \ldots, 0, \eps)
\]
with a paramter $\eps \geq 0$. If \reff{mom-las:put} is solved by
interior-point methods, then typically a value $\eps >0$ is returned.
For instance, for $k=2$, a numerical value $0.7908$ of $\eps$
is returned when \reff{mom-las:put} is solved by the SDP package {\tt SeDuMi} \cite{sedumi}.
Clearly, \reff{con:fec} is violated for $y^*(\eps)$ with $\eps>0$,
while \reff{cd:flat-trun} is always satisfied for $y^*(\eps)$
with every $\eps\geq 0$ and $1\leq t \leq k-1$.
This shows that flat truncation is a more suitable condition
than flat extension in certifying convergence of Lasserre's hierachy.
\qed
\end{exm}

The rank condition \reff{cd:flat-trun} was used in \cite{LasBok,Lau}
as a sufficient condition to verify exactness of Lasserre's relaxations
in polynomial optimization and generalized problems of moments.
When the feasible set $K$ is defined by
polynomial equalities $h_1(x)=\cdots=h_s(x)=0$,
if the equations $h_i(x)=0$ have finitely many zeros,
then Lasserre's hierarchy has finite convergence
(cf. \cite[\S3]{Lau07} or \cite[\S6.5]{Lau}),
and for $k$ big enough every tms $y$ that is feasible for the resulting \reff{mom-las:put}
has a flat truncation (cf. \cite[Prop.~4.6]{LLR08} or \cite[Theorem~6.20]{Lau}).
For general polynomial optimization problems, there are no such
results in the existing literature.

As we have seen earlier,
for an optimizer of \reff{mom-las:put} to have a flat truncation,
a necessary condition is $f_k = f_{min}$
(suppose \reff{las-rex:put} and \reff{mom-las:put} has no duality gap).
That is, flat truncation is a sufficient condition for
Lasserre's hierarchy to have finite convergence.
Thus, one is wondering whether flat truncation is also necessary:
if Lasserre's hierarchy converges in finitely many steps, does {\it every}
\footnote{If $f_k = f_{min}$ and \reff{pop:K}
has a minimizer, say $x^*$, then \reff{mom-las:put} always
has an optimizer that is flat, e.g., $[x^*]_{2k}$.
Thus, it is more sensible to ask whether every optimizer has a flat truncatioin.}
optimizer of \reff{mom-las:put} have a flat truncation?
If so, flat truncation can be used as a certificate
for checking its finite convergence
and some minimizers of \reff{pop:K} can also be obtained.
This issue was addressed very little in the literature.

Another important issue in applications is how to certify
asymptotic convergence of Lasserre's hierarchy.
Under the archimedean condition, we know
$f_k \to f_{min}$ as $k \to \infty$ (cf. \cite{Las01}),
but the convergence to the set of global minimizers of \reff{pop:K} is not known very well.
When \reff{pop:K} has a unique global minimizer $u^*$,
Schweighofer \cite{Swg05} showed that
the subvector consisting of linear moments of a nearly optimizer
of \reff{mom-las:put} converges to $u^*$ as $k\to\infty$.
However, for the more general case that \reff{pop:K} has
more than one global minimizer, there are no such results
in the existing literature.
It is possible that $f_k = f_k^* < f_{min}$ for every $k$,
i.e., Lasserre's hierarchy may have no finite convergence,
as implied by Scheiderer's work \cite{Sch99}.
For instance, for the problem of minimizing the Motzkin polynomial over the unit ball,
the Lasserre's hierarchy does not converge within finitely many steps
(cf. \cite[Example~5.3]{Nie:JacSDP}).
In such cases, one should not expect any minimizer of
\reff{mom-las:put} to have a flat truncation.
However, how about the limit points of truncations of minimizers of
\reff{mom-las:put} as $k\to \infty$?
Is every such limit point flat, or does it have a flat truncation?
There is very little work on this issue.

\subsection{Contributions}

This paper focuses on proving a certificate
for checking convergence of Lasserre's hierarchy.
We assume \reff{pop:K} has a nonempty set $S$ of
global minimizers and its cardinality $|S|$ is finite.
Our main result is that flat truncation can generically
serve as such a certificate.

First, we study how to certify finite convergence.
For Lasserre's hierarchy of \reff{las-rex:put}
(it is also called a Putinar type one),
we show that: it has finite convergence
if and only if every minimizer of \reff{mom-las:put} has a flat truncation
when $k$ is sufficiently large,
under some generic conditions (e.g., Assumption~\ref{crt-pt<inf}).
For Schm\"{u}dgen type Lasserre's hierarchy,
which is a refined version of \reff{las-rex:put}
by using cross products of $g_j$ (cf. \reff{sos-rex:shmg}),
we show the same conclusion holds under some slightly weaker conditions.
This will be shown in Section~\ref{sec:exact-las}.

Second, we study how to certify asymptotical convergence.
Let $\{ y^{(k)} \}_{k=1}^{\infty}$ be a sequence of asymptotically optimal solutions
of \reff{mom-las:put}. We prove that: under the archimedean condition,
for any fixed order $t \geq \max\{d_f,d_g+|S|-1\}$,
the truncated sequence $\{y^{(k)}|_{2t}\}_{k=1}^{\infty}$ is bounded,
and its every limit point is flat, i.e.,
flat truncation is asymptotically satisfied for both Lasserre's hierarchies.
The same conclusion holds for the Schm\"{u}dgen type one when $K$ is compact
(the archimedean condition is not required).
This will be shown in Section~\ref{sec:cvg-las}.

Third, we show that flat truncation can be used
as a certificate to check exactness of standard SOS relaxations
and Jacobian SDP relaxations.
This will be shown in Section~\ref{sec:spec-cas}.

Last, we make some discussions about flat truncation
in Section~5.

\section{Certifying Finite Convergence}
\label{sec:exact-las}
\setcounter{equation}{0}

As we have seen in \S\ref{subsec:flat},
flat truncation is a sufficient condition
for Lasserre's hierarchy of \reff{las-rex:put} to converge in finitely many steps.
In this section, we show that flat truncation is also
a necessary condition in the generic case. Thus,
it could generically serve as a certificate for checking
finite convergence of \reff{las-rex:put}.
Lasserre's hierarchy of \reff{las-rex:put} is also called a Putinar type one,
since it uses Putinar's Positivstellensatz \cite{Put} in representing positive polynomials.
A refining of \reff{las-rex:put} is Schm\"{u}dgen type Lasserre's hierarchy,
which uses cross products of the constraining polynomials. Similarly,
flat truncation could also be generically used as a certificate
for checking finite convergence of the Schm\"{u}dgen type one.
We present the results in two separate subsections.

\subsection{Putinar Type Lasserre's relaxation}

The quadratic module generated by the tuple $g$ is
\[
Q(g) \quad := \quad \bigcup_{k=1}^\infty Q_k(g).
\]
The archimedean condition requires that there exists $\phi\in Q(g)$
such that the inequality $\phi(x)\geq 0$ defines a compact set in $x$.
Note that the archimedean condition implies the feasible set $K$ is compact.
The convergence for Lasserre's hierarchy of \reff{las-rex:put}
is based on Putinar's Positivstellensatz \cite{Put}:
if a polynomial $p$ is positive on $K$ and
the archimedean condition holds, then $p \in Q(g)$.
To certify finite convergence of Putinar type Lasserre's hierarchy,
we need the following assumption on $f,g_1,\ldots,g_m$.

\begin{ass} \label{crt-pt<inf}
There exists $\rho \in Q(g)$ such that
for every $J  \subseteq \{1,\ldots, m\}$ and
\[
\mc{V}_J := \{ x\in \re^n: \,   g_j(x)=0 \,\, (\forall j \in J), \,
\rank \, G_J(x) \leq |J| \},
\]
the intersection $\mc{V}_J \cap \mc{M} \cap \mc{P}$ is finite.
Here, denote $J=\{j_1,\ldots,j_l\}$,
\[
G_J(x) := \bbm \nabla f(x) & \nabla g_{j_1}(x) & \cdots &
\nabla g_{j_l}(x)\ebm,
\]
and $\mc{M} := \{x \in \re^n : \, f(x) = f_{min} \}$,
$\mc{P} := \{x \in \re^n : \, \rho(x) \geq 0 \}$.
\end{ass}

Assumption~\ref{crt-pt<inf} requires that for every $J$ the optimization problem
\[
\min  \quad f(x) \quad  \mbox{ s.t. } \quad g_j(x) = 0 \,\,(j\in J)
\]
has finitely many critical points lying on $\mc{M} \cap \mc{P}$
(if $u$ is a critical point of the above,
then $\rank\, G_J(u) \leq |J|$, cf.~\cite[\S2]{NR09}).
Let $S$ be the set of global minimizers of \reff{pop:K} and
\[
\mc{V}:= \cup_J \mc{V}_J.
\]
Clearly, $S\subseteq \mc{V}$, and Assumption~\ref{crt-pt<inf}
implies that $S$ is finite.

We would like to remark that Assumption~\ref{crt-pt<inf} is {\it generically} true.
(Here, we say a property is generically true if it holds
in the space of data except a set of Lebesgue measure zero, cf. \cite{NR09}).
As shown by Proposition~2.1 of \cite{NR09},
if $f$ and every $g_j$ have generic coefficients,
then $\mc{V}_J$ is finite for every index set $J$,
and Assumption~\ref{crt-pt<inf} is satisfied by simply choosing $\rho = 0$.

As a concrete example, Example~\ref{emp:FTneFE} satisfies Assumption~\ref{crt-pt<inf}.
It has $m=2$ inequalities.  Let $f,g_1,g_2$ be the objective,
first and second constraining polynomials respectively,
then clearly
\[
\mc{V}_{\emptyset} = \{ 0\}, \quad
\mc{V}_{ \{1\} } = \{ 0\}, \quad
\mc{V}_{\{2\} } = \{ 1\}, \quad
\mc{V}_{ \{1,2\} } = \emptyset.
\]
They are all finite sets, so Assumption~\ref{crt-pt<inf} is satisfied.
Example~\ref{emp:FTneFE} is in the generic case
where Assumption~\ref{crt-pt<inf} holds.

Assumption~\ref{crt-pt<inf} also holds in the following cases:
\bit

%

\item Suppose $S$ is finite, $\mc{V}\cap \mc{M}$ is infinite,
but $(\mc{V}\cap \mc{M})\backslash K$ is contained in a compact set $T$
not intersecting $K$. Then, there exists $\dt>0$ such that
\[
dist(x,K) \geq \dt \,\, \forall \, x \in T, \qquad
dist(x,K) =0 \,\, \forall \, x \in K.
\]
(Here $dist(x,K) = \min_{u\in K} \|x-u\|_2$ and $\|\cdot\|_2$
denotes the standard $2$-norm.)
The function $\tau(x):=dist(x,K)-\dt/2$ is continuous in $x$.
Then, by the compactness of $K$ and $T$,
there exists a polynomial $\eta$
that is an approximation of $\tau$ and satisfies
\[
\eta(x) \geq \dt/4 \,\, \forall \, x \in T, \qquad
\eta(x) \leq -\dt/4 \,\, \forall \, x \in K.
\]
Clearly, $-\eta$ is positive on $K$. Assume the archimedean condition
holds for $g$, then $-\eta \in Q(g)$ by Putinar's Positivstellensatz.
Assumption~\ref{crt-pt<inf} is satisfied by choosing $\rho = -\eta$,
because $\mc{V} \cap \mc{M} \cap \mc{P} = S$ is finite.

\item  Suppose $\mc{V} \cap \mc{M}$ is unbounded
but, except finitely many points, lies on a real algebraic variety
\[  \{x\in \re^n: w_1(x)=\cdots=w_r(x)=0\} \] not intersecting $K$.
Suppose the archimedean condition holds for $g$, then $K$ is compact
and there exists $\eps>0$ such that
the polynomial $w:=w_1^2+\cdots+w_r^2-\eps$ is positive on $K$.
By Putinar's Positivstellensatz, $w \in Q(g)$.
Assumption~\ref{crt-pt<inf} is satisfied by choosing $\rho = w$.

%
%

\eit

Our main result of this subsection is the following theorem.

\begin{theorem}  \label{thm:flat-put}
Suppose Assumption~\ref{crt-pt<inf} holds,
the set $S$ of global minimizers of \reff{pop:K} is nonempty,
and for $k$ big enough the optimal value of \reff{las-rex:put} is achievable
and there is no duality gap between \reff{las-rex:put} and \reff{mom-las:put}.
Then, Lasserre's hierarchy of \reff{las-rex:put} has finite convergence
if and only if every minimizer of \reff{mom-las:put} has a flat truncation
for $k$ sufficiently large.
\end{theorem}

\begin{remark} \label{rmk:ftn-put}
In Theorem~\ref{thm:flat-put},
1) if $K$ has nonempty interior, then for every order $k$
\reff{las-rex:put} achieves its optimum,
and it has no duality gap;
2) if finite convergence occurs, then \reff{mom-las:put}
always has a minimizer for $k$ big enough
(e.g., $[x^*]_{2k}$ is one for any $x^* \in S$),
and there is no duality gap;
3) when $k$ is big enough, for every minimizer $y^*$ of \reff{mom-las:put},
$y^*|_{2k-2}$ is always flat, as shown in the proof later;
4) we do not need to assume $K$ is compact or the archimedean condition holds.
\end{remark}

\begin{remark} \label{inf=>nonFlat}
As pointed out in Laurent's survey \cite[\S6.6]{Lau},
we could get the following properties about flat truncation.
1) In Theorem~\ref{thm:flat-put}, Assumption~\ref{crt-pt<inf} implies $S$ is finite.
If $\rank\, M_k(y^*)$ is maximum over the optimizers
of \reff{mom-las:put}, then for any flat truncation of $y^*$, say, $y^*|_{2t}$,
$\rank\,M_t(y^*)=|S|$. This means that all the minimizers of \reff{pop:K}
could be extracted from $y^*|_{2t}$.
2) When $S$ is infinite and $\rank\, M_k(y^*)$ is maximum over the optimizers
of \reff{mom-las:put}, then $y^*$ could not have a flat truncation.
3) When \reff{las-rex:put} and \reff{mom-las:put} are solved by primal-dual
interior-point algorithms,
a minimizer $y^*$ near the analytic center of the face of optimizers of \reff{mom-las:put}
is usually returned and $\rank\, M_k(y^*)$ is typically maximum.
Therefore, if the conditions of Theorem~\ref{thm:flat-put} are satisfied,
by solving a sufficient high order Lasserre's relaxation via interior point methods,
then we could typically find all minimizers of \reff{pop:K}
if $S$ is finite, and flat truncation could not be satisfied if $S$ is infinite.
\end{remark}

To prove Theorem~\ref{thm:flat-put}, we need some properties about the kernels
of moment and localizing matrices.
(Given a polynomial $p$, we also denote its coefficient vector by $p$,
for convenience of notations.)
For a localizing matrix $L_{h}^{(k)}(y)$, recall that
if $\deg(hp^2) \leq 2k$ then
\[
p^T \cdot L_{h}^{(k)}(y) \cdot p \quad = \quad
\mathscr{L}_y( h p^2).
\]
If $L_{h}^{(k)}(y) \cdot p =0$, we say $p \in \ker L_{h}^{(k)}(y)$.
When $M_k(y) = L_{1}^{(k)}(y)$ is a moment matrix ,
we similarly say $p \in \ker M_k(y)$ if $M_k(y) \cdot p =0$.

\begin{lem}  \label{mom-ker-limit}
Let $y \in \mathscr{M}_{2k}, h \in \re[x]$
be such that $L_{h}^{(k)}(y)\succeq 0$.
\bit
\item [i)] (\cite{LLR08,Lau}) Suppose $M_k(y)\succeq 0$. Let $p,q \in \re[x]$.
If $\deg(pq) \leq k-1$ and $q \in \ker M_k(y)$, then $pq \in \ker M_k(y)$.
If $q^\ell \in \ker M_k(y)$ and $2\lceil \ell/2 \rceil \deg(q) \leq k-1$, then $q \in \ker M_k(y)$.

\item [ii)](\cite{HN4})
Let $s$ be an SOS polynomial with $\deg(hs) \leq 2k$.
Then, $\langle hs, y \rangle \geq 0$.
If $\langle hs, y \rangle =0$,
then for any $\phi \in \re[x]_{2\ell}$ with $\deg(hs) + 2\ell \leq 2k-2$
we have $\langle hs\phi, y \rangle =0$.

\item [iii)] Let $\{p_j\}_{j=1}^\infty \subset \re[x]$ be a sequence such that
each $\deg(h p_j^2) \leq 2k$ and
\[
L_{h}^{(k)}(y) \cdot p_j \to 0 \quad \mbox{ as } \quad j \to \infty.
\]
If $q\in \re[x]$ and every $\deg(p_j q) \leq k-\lceil \deg(h)/2 \rceil-1$,
then
\[
L_{h}^{(k)}(y)\cdot (p_j q) \to 0 \quad \mbox{ as } \quad j \to \infty.
\]

\item [iv)]
Let $\{s_j\}_{j=1}^\infty$ be a sequence of SOS polynomials such that
\[
\deg(h s_j) < 2k \,\,( \forall \, j), \quad \mbox{ and } \quad
\langle hs_j, y \rangle \to  0 \mbox{ as } \, j \to \infty.
\]
If $\phi \in \re[x]_{2\ell}$ and every $\deg(h s_j)+2\ell \leq 2k-2$, then
\[
\langle h s_j \phi, y \rangle  \to 0 \quad \mbox{ as } \quad j \to \infty.
\]
\eit
\end{lem}
\begin{proof}
i) The first part is basically from \cite[Lemma~5.7]{Lau} or \cite[Lemma~21]{Lau07}.
For the second part, if $\ell$ is even, the result follows Lemma~3.9 of \cite{LLR08};
if $\ell$ is odd, then $q^{\ell+1} \in \ker M_k(y)$ from the first part,
and the result is still true.

ii)
Write $s=\sum_i s_i^2$, then
\[
\langle hs, y \rangle = \sum_i \langle hs_i^2, y \rangle
= \sum_i s_i^T L_{h}^{(k)}(y) s_i \geq 0.
\]
Now, assume $\langle hs, y \rangle =0$.
One could always write $\phi = \sum_j p_jq_j$
with all $p_j, q_j \in \re[x]_\ell$.
Then, we have
\[
\langle hs, y \rangle = \sum_i s_i^T L_{h}^{(k)}(y) s_i, \quad
\langle hs\phi, y \rangle =  \sum_{i,j} (s_ip_j)^T L_{h}^{(k)}(y) (s_iq_j).
\]
Since $L_{h}^{(k)}(y)\succeq 0$, $\langle hs, y \rangle =0$
implies every $L_{h}^{(k)}(y) s_i=0$.
Then, by Lemma~2.3 of \cite{HN4}, every $L_{h}^{(k)}(y) (s_iq_j)=0$
and $(s_ip_j)^T L_{h}^{(k)}(y) (s_iq_j)=0$. So, $\langle hs\phi, y\rangle =0$.

iii)
Let $r = k - \lceil \deg(h)/2 \rceil$.
By a simple induction on $\deg(q)$,
it suffices to prove the lemma for $q = x_i \,(1\leq i\leq n)$.
Note that every $\deg(p_j)+1 \leq r -1$. Let
\[
u^{(j)}:= L_{h}^{(k)}(y) \cdot p_j,
\qquad v^{(j)} := L_{h}^{(k)}(y) \cdot (p_j \cdot x_i).
\]
They are indexed by integral vectors $\af \in \N^n$, and could also be expressed as
\[
u^{(j)} = \mathscr{L}_y \big( p_j \cdot h \cdot \mathfrak{m}_r \big), \quad
v^{(j)} =  \mathscr{L}_y \big(x_i \cdot p_j \cdot h \cdot \mathfrak{m}_r \big).
\]
(Here, $\mathfrak{m}_r$ is the vector of monomials of degrees $\leq r$.)
Then $u^{(j)} \to 0$ implies
\[
\underset{j\to\infty}{\lim} \mathscr{L}_y \big( p_j \cdot h \cdot x^\bt \big) = 0
\quad \mbox{ if } \quad  |\bt|\leq r.
\]
The $\af$-th entry of $v^{(j)}$ is
\[
v_\af^{(j)} := \mathscr{L}_y \big( p_j \cdot h \cdot x^{\af} \cdot x_i \big).
\]
If $\deg(x^{\af} \cdot x_i ) \leq r$, then $ v_\af^{(j)} = u_{\tilde{\af}}^{(j)}  \to 0$
($\tilde{\af}$ is the exponent of $x^{\af} \cdot x_i$). So,
\[
\underset{j\to\infty}{\lim} v_\af^{(j)} = 0 \quad
\mbox{ if } \, |\af| \leq r-1.
\]
This implies that (note $L_{h}^{(k-1)}(y)$ is a
leading principal submatrix of $L_{h}^{(k)}(y)$)
\[
\underset{j\to\infty}{\lim} \left( L_{h}^{(k-1)}(y) \right) \cdot (p_j\cdot x_i) =
\underset{j\to\infty}{\lim}
\mathscr{L}_y \big(x_i \cdot p_j \cdot h \cdot \mathfrak{m}_{r-1} \big) = 0.
\]
Since $L_h^{(k-1)}(y)$ is symmetric and positive semidefinite,
there exist orthogonal vectors $a_1, \ldots, a_t$ such that
$
L_h^{(k-1)}(y) = a_1a_1^T + \cdots + a_t a_t^T.
$
So, we have
\[
a_\ell^T (p_j\cdot x_i) \to 0 \quad \mbox{ as } j \to \infty,
\quad \ell = 1,\ldots, t.
\]
By the above, it holds that
\[
\underset{j\to\infty}{\lim}
(p_j\cdot x_i)^T  \cdot \left( L_{h}^{(k-1)}(y) \right) \cdot (p_j\cdot x_i) =
\underset{j\to\infty}{\lim} \sum_{\ell=1}^t  \big( a_\ell^T (p_j\cdot x_i)  \big)^2
= 0.
\]
When $|\af| > r-1$, the coefficient of
$x^\af$ in the polynomial $p_j\cdot x_i$ is zero. So
\[
(p_j \cdot x_i)^T \cdot \left( L_{h}^{(k)}(y) \right) \cdot (p_j \cdot x_i) =
(p_j \cdot x_i)^T \cdot \left( L_{h}^{(k-1)}(y) \right) \cdot (p_j \cdot x_i).
\]
Hence, we also have
\[
\underset{j\to\infty}{\lim} (p_j\cdot x_i)^T
\cdot \left( L_{h}^{(k)}(y) \right) \cdot (p_j\cdot x_i) = 0.
\]
Since $L_{h}^{(k)}(y)\succeq 0$, the above implies that
\[
\underset{j\to\infty}{\lim} \left( L_{h}^{(k)}(y) \right) \cdot (p_j\cdot x_i) = 0.
\]

iv) Write $s_j = \sum_t p_{j,t}^2$
(the length of summation is at most $\binom{n+k}{k}$, since $\deg(s_j)\leq 2k$).
By item ii), $\langle hs_j, y \rangle  \geq \langle hp_{j,t}^2, y \rangle  \geq 0$.
So, $\langle hs_j, y \rangle \to 0$ implies that for each $t$
\[
\langle h p_{j,t}^2, y \rangle = (p_{j,t})^T L_h^{(k)}(y) p_{j,t} \to 0
\quad \mbox{ as } \, j \to \infty.
\]
Since $L_h^{(k)}(y) \succeq 0$, we have $L_h^{(k)}(y) p_{j,t} \to 0$ as $j\to \infty$.
Like in item ii), one could write $\phi = \sum_i u_i v_i$ with all $u_i, v_i \in \re[x]_\ell$.
By item iii), it holds that for all $t,i$
\[
L_h^{(k)}(y) (p_{j,t} u_i) \to 0,\,
L_h^{(k)}(y) (p_{j,t} v_i) \to 0 \quad \mbox{ as } \, j \to \infty.
\]
So, we also have
\[
\left(L_h^{(k)}(y)\right)^{\half} (p_{j,t} u_i) \to 0,\,
\left(L_h^{(k)}(y)\right)^{\half} (p_{j,t} v_i) \to 0 \quad \mbox{ as } \, j \to \infty.
\]
(The $X^{\half}$ denotes the standard matrix square root of
a symmetric positive semidefinite matrix $X$.)
The above implies that
\[
(p_{j,t} u_i)^T L_h^{(k)}(y) (p_{j,t} v_i)
= (p_{j,t} u_i)^T \left(L_h^{(k)}(y)\right)^{\half} \cdot
\left(L_h^{(k)}(y)\right)^{\half}  (p_{j,t} v_i) \to 0
\]
and
\[
\langle h s_j \phi, y \rangle = \sum_{t,i} \langle hp_{j,t}^2u_iv_i, y \rangle  =
\sum_{t,i} (p_{j,t} u_i)^T L_h^{(k)}(y) (p_{j,t} v_i) \to 0,
\]
as $j \to \infty$.
\end{proof}

\noindent
{\it Proof of Theorem~\ref{thm:flat-put} } \,
The sufficiency of flat truncation was observed in \S 1.2,
if there is no duality gap between \reff{las-rex:put} and \reff{mom-las:put}.
We only need to prove its necessity.
Suppose $f_{k_0}=f_{min}$ for $k_0$ big enough.
Since \reff{las-rex:put} has a maximizer when its order is big enough,
we could assume \reff{las-rex:put} achieves its optimum $f_{min}$ for order $k_0$
(otherwise increase $k_0$),
i.e., $f-f_{min} \in Q_{k_0}(g)$.
So, there exist SOS polynomials $s_0, s_1, \ldots, s_m$ such that
every $\deg(g_is_i) \leq 2k_0$ and
\be \label{f-min:Q(g)}
f - f_{min} = s_0 + g_1s_1+\cdots+ g_ms_m.
\ee
Let $y^*$ be an arbitrary minimizer of \reff{mom-las:put}
(it always has one when $k\geq k_0$, e.g., $[x^*]_{2k}$, for any $x^*\in S$). Clearly,
$
\langle f, y^*\rangle = f_{\min}.
$
Let $C$ be the semialgebraic set defined as ($\rho$ is from Assumption~\ref{crt-pt<inf})
\be \label{def:C-rt}
C = \{ x \in \re^n: \, s_0(x) = g_1(x)s_1(x) = \cdots = g_m(x) s_m(x) =0, \rho(x) \geq 0 \}.
\ee
We complete the proof in three steps.

\bigskip
\noindent
{\bf Step~1} \, We show that $C$ is a finite set.
In the identity \reff{f-min:Q(g)}, differentiating its both sides
in $x$ results in (note $g_0 = 1$)
\[
\nabla f =  \sum_{i=0}^m \Big( s_i \cdot \nabla g_i
+ g_i \nabla s_i \Big).
\]
Choose an arbitrary $u\in C$, then clearly $u \in \mc{M} \cap \mc{P}$.
Let $J(u)=\{i: g_i(u) = 0 \}$.
Note that for every $i \not\in J(u)$, $s_i(u)=0$, and it implies
$\nabla s_i(u)=0$ (because $s_i$ is SOS and $u$ is a minimizer of $s_i$).
So, by the above, it holds that
\[
\nabla f(u) =  \sum_{i \in J(u)} s_i(u) \cdot \nabla g_i(u).
\]
Hence, $u\in \mc{V}_{J(u)} \cap \mc{M} \cap \mc{P}$.
Since there are at most $2^m$ active index sets like $J(u)$,
by Assumption~\ref{crt-pt<inf}, $C$ must be finite.

\medskip
\noindent
{\bf Step~2} \, We show that every generator of the vanishing ideal
of $C$ belongs to the kernel of $M_k(y^*)$ for $k$ big enough.
Since $C$ is finite, its vanishing ideal
\[
I(C) \, := \, \{ p\in \re[x]:\, p(u) = 0 \quad \forall \, u\in C\}
\]
is zero dimensional.
Let $\{h_1,\ldots, h_r\}$ be a Grobner basis of $I(C)$
with respect to a total degree ordering. Clearly, each $h_i$ vanishes on $C$.
By Real Nullstellensatz (cf. Corollary~4.4.3 of \cite{BCR}),
there exist $\ell \in \N $, $\phi_1,\ldots,\phi_m \in \re[x]$,
and $\varphi \in Q(\rho)$ (the quadratic module generated by the single polynomial $\rho$,
which is also equal to the prepordering generated by $\rho$) such that
\[
h_i^{2\ell} + \varphi + g_1s_1 \phi_1 + \cdots + g_ms_m \phi_m = 0.
\]
Applying the Riesz functional $\mathscr{L}_{y^*}$ to the above
(suppose $2k$ is bigger than the degrees of all the products there), we get
\be \label{Ly-ReNul=0}
\langle h_i^{2\ell}, y^* \rangle  + \langle \varphi, y^* \rangle +
\sum_{j=1}^m  \langle g_js_j\phi_j, y^* \rangle = 0.
\ee
Applying $\mathscr{L}_{y^*}$ to \reff{f-min:Q(g)} results in
(note $\langle f, y^* \rangle = f_{min}$)
\be
0 = \langle f - f_{min}, y^* \rangle =
\sum_{j=0}^m  \langle g_js_j, y^* \rangle.
\ee
Since each $s_j$ is SOS, every $\langle g_js_j, y^* \rangle\geq 0$,
by item ii) of Lemma~\ref{mom-ker-limit}.
Thus, from the above, we know every
$
\langle g_js_j, y^* \rangle  =0.
$
Again, by item ii) of Lemma~\ref{mom-ker-limit},
if $2k > 2 +\deg(g_js_j\phi_j)$,  then every
$
\langle g_js_j\phi_j, y^* \rangle =0.
$
So, from \reff{Ly-ReNul=0}, we can get
\be \label{<h2l,y>+<vphi,y>=0}
\langle h_i^{2\ell}, y^* \rangle  + \langle \varphi, y^* \rangle =0.
\ee
Since $h_i^{2\ell}$ is SOS, we similarly have
$\langle h_i^{2\ell}, y^* \rangle \geq 0$.
Since $Q(\rho) \subset Q(g)$, $\varphi \in Q(g)$
and one could write
$
\varphi = \sum_{j=0}^m g_j\sig_j
$
with each $\sig_j$ being SOS. Hence,
\[
\langle \varphi, y^* \rangle  =
\langle g_0 \sig_0, y^* \rangle + \langle g_1\sig_1, y^* \rangle
+ \langle g_m\sig_m, y^* \rangle \geq 0,
\]
by item ii) of Lemma~\ref{mom-ker-limit}.
Then \reff{<h2l,y>+<vphi,y>=0}
implies $\langle h_i^{2\ell}, y^* \rangle  = 0$,
i.e., $h_i^\ell \in \ker M_k(y^*)$.
By item i) of Lemma~\ref{mom-ker-limit} and
$M_k(y^*) = L_{g_0}^{(k)}(y^*)\succeq 0$, if $k$ is big enough, we get
\[
h_i \in \ker M_k(y^*).
\]

\medskip
\noindent
{\bf Step~3} \, It's enough to show that $y^*|_{2k-2}$ is flat.
Since $C$ is finite, the quotient space $\re[x]/I(C)$ is finitely dimensional.
Let $\{b_1, \ldots, b_L\}$ be a standard basis of $\re[x]/I(C)$.
Then, for every $\af \in \N^n$, we can write
\[
x^\af = \eta(\af) + \sum_{i=1}^r \theta_i h_i, \quad
\deg(\theta_i h_i) \leq |\af|, \quad
\eta(\af) \in \mbox{span}\{b_1,\ldots, b_L\}.
\]
Because every $h_i \in \ker M_k(y^*)$, we know
\[
\theta_i h_i \in \ker M_k(y^*) \quad  \mbox{ if }  |\af| \leq k-1,
\]
by item i) of Lemma~\ref{mom-ker-limit}.
Thus,
\[
x^\af - \eta(\af) \in \ker M_k(y^*) \quad \mbox{ if } |\af| \leq k-1.
\]
Set $d_b:=\max_j \deg(b_j)$. Then, every $\af$-th column ($d_b+1 \leq |\af| \leq k-1$)
of $M_k(y^*)$ is a linear combination
of $\bt$-th columns of $M_k(y^*)$ with $|\bt|\leq d_b$, so
\[
\rank \, M_{d_b}(y^*) = \rank\, M_{t}(y^*), \quad
\, t=d_b+1,\ldots, k-1.
\]
Hence, if $k-1-d_g \geq d_b$, then
\[
\rank \, M_{k-1-d_g}(y^*) = \rank\, M_{k-1}(y^*).
\]
That is, $y^*|_{2k-2}$ is flat and $y^*$ has a flat truncation,
when $k$ is big enough.
\qed

\subsection{Schm\"{u}dgen type Lasserre's relaxation}

Now we consider Schm\"{u}dgen type Lasserre's hierarchy,
which refines \reff{las-rex:put} as:
\be \label{sos-rex:shmg}
 \max \quad \gamma \quad
\mbox{ s.t. }  f - \gamma \in Pr_k(g).
\ee
The above $Pr_k(g)$ denotes the $k$-th truncated
preordering generated by the tuple $g$
(denote $g_\nu := g_1^{\nu_1}\cdots g_m^{\nu_m}$):
\[ 
Pr_k(g) := \left\{ \left.\sum_{\nu \in \{0,1\}^m } g_\nu \sig_\nu \right|
\sig_\nu \mbox{ is SOS}, \, \deg(g_\nu \sig_\nu) \leq 2k \mbox{ for every } \nu
\right\}.
\]
The dual optimization problem of \reff{sos-rex:shmg} is (cf. \cite{Las01, LasBok})
\be \label{mom-sdp:schmg}
\left\{\baray{rl}
\underset{ y \in \mathscr{M}_{2k} }{\min} &  \langle f, y \rangle \\
\mbox{s.t.}& L_{g_\nu}^{(k)}(y) \succeq 0 \,(\nu\in \{0,1\}^m), \,
\langle 1, y \rangle = 1.
\earay \right.
\ee
Let $\widetilde{f}_k$ and $\widetilde{f}_k^*$ be the optimal values of
\reff{sos-rex:shmg} and \reff{mom-sdp:schmg} respectively, for a given order $k$.
By weak duality, $\widetilde{f}_k \leq \widetilde{f}_k^*$ for every $k$.
If $K$ has nonempty interior, then $\widetilde{f}_k = \widetilde{f}_k^*$,
i.e., there is no duality gap.
Clearly, every $\widetilde{f}_k^* \leq f_{min}$.
Both sequences $\{\widetilde{f}_k\}$ and $\{\widetilde{f}_k^*\}$
are monotonically increasing.
The relaxation \reff{sos-rex:shmg} is stronger than \reff{las-rex:put},
because $Q_k(g) \subseteq Pr_k(g)$, so we have $f_k \leq \widetilde{f}_k$ for every $k$.
By Schm\"{u}dgen's Positivstellensaz
(if $K$ is compact and a polynomial $p$ is positive on $K$,
then $p \in Pr_{\ell}(g)$ for some $\ell$, cf. \cite{Smg}),
$\widetilde{f}_k$ converges to $f_{min}$,
and so does $\widetilde{f}_k^*$, when $K$ is compact.
If $\widetilde{f}_{k_1}=f_{min}$ for some order $k_1$, we say Schm\"{u}dgen type
Lasserre's hierarchy has {\it finite convergence}.

Suppose $y^*$ is an optimizer of \reff{mom-sdp:schmg}.
If $y^*$ has a flat truncation, say $y^*|_{2t}$, then, as shown in \S1.2,
one could not only extract $\rank\, M_t(y^*)$ global optimizers of \reff{pop:K} from $y^*$,
but also get a certificate for $\widetilde{f}_k=f_{min}$
if there is no duality gap between \reff{sos-rex:shmg} and \reff{mom-sdp:schmg}.
So, flat truncation is also a sufficient condition for
Schm\"{u}dgen type Lasserre's hierarchy to have finite convergence.
Is it also a necessary condition?
If so, flat truncation could serve as a certificate
for checking finite convergence of the hierarchy of \reff{sos-rex:shmg}.
Like for the Putinar type one, this is also generically true.
A similar result like Theorem~\ref{thm:flat-put} holds,
and weaker conditions are required.

\begin{theorem} \label{flatrun:exc-rlx}
Suppose the set $S$ of global minimizers of \reff{pop:K}
is nonempty and finite, and
there is no duality gap between \reff{sos-rex:shmg} and \reff{mom-sdp:schmg}
for $k$ big enough.
Then, Schm\"{u}dgen type Lasserre's hierarchy of \reff{sos-rex:shmg}
has finite convergence if and only if every minimizer of
\reff{mom-sdp:schmg} has a flat truncation for $k$ sufficiently large.
\end{theorem}

\begin{remark} \label{rmk:schmg-las}
The comments in Remarks~\ref{rmk:ftn-put} and \ref{inf=>nonFlat} for Theorem~\ref{thm:flat-put}
are all applicable to Theorem~\ref{flatrun:exc-rlx}.
Here, we point out some differences:
1) Theorem~\ref{flatrun:exc-rlx} does not require
the optimum of \reff{sos-rex:shmg} to be achievable.
2) Assumption~\ref{crt-pt<inf} is slightly stronger than that
$S$ is finite, although they are both generically true.
3) If it occurs that $Q_k(g)=Pr_k(g)$ (e.g., this is the case when \reff{pop:K}
has equality constraints and/or a single inequality constraint),
then in Theorem~\ref{thm:flat-put}
Assumption~\ref{crt-pt<inf} could be replaced by $|S|<\infty$
and \reff{las-rex:put} is not required to achieve its optimum.
\end{remark}

\medskip
\noindent
{\it Proof of Theorem~\ref{flatrun:exc-rlx} } \,
Like in the proof of Theorem~\ref{thm:flat-put},
we only need to prove the necessity of flat truncation.
Suppose $\widetilde{f}_{k_1}=f_{min}$ for some order $k_1$,
then for every $\eps >0$ we have
$
f - f_{min} +\eps \in Pr_{k_1}(g).
$
Write
\be \label{f-min+eps:sch}
f - f_{min} +\eps = \sum_{ \nu \in \{0,1\}^m } \sig_\nu^\eps \cdot g_{\nu}
\ee
for some SOS polynomials $\sig_\nu^\eps$ with $\deg(\sig_\nu^\eps g_{\nu}) \leq 2k_1$.
Note that as $\eps \to 0$ some coefficients of $\sig_\nu^\eps$
might go to infinity while its degree is bounded.
Let $y^*$ be an arbitrary minimizer of \reff{mom-sdp:schmg}
(if $k\geq k_1$, \reff{mom-sdp:schmg} always has one,
e.g., $[x^*]_{2k}$, for any $x^*\in S$). Then
$\mathscr{L}_{y^*}(f) =\widetilde{f}_k = f_{min}$.
Applying $\mathscr{L}_{y^*}$ to \reff{f-min+eps:sch},
we get
\be \label{eps=sum-nu}
\eps = \sum_{ \nu \in \{0,1\}^m } \langle \sig_\nu^\eps g_{\nu}, y^* \rangle.
\ee
Since every $\sig_\nu^\eps$ is SOS, by item ii) of Lemma~\ref{mom-ker-limit},
$\langle \sig_\nu^\eps g_{\nu}, y^* \rangle \geq 0$ and
\be \label{Lgyp-->0}
\lim_{ \eps \to 0} \, \langle \sig_\nu^\eps g_{\nu}, y^* \rangle = 0.
\ee
We complete the proof in three steps.

\medskip
\noindent
{\bf Step 1 }\,
The set $S$ is finite, so its vanishing ideal
\[
I(S) \quad = \quad \{ p\in \re[x]:\, p(u) = 0 \quad \forall \, u\in S\}
\]
is zero dimensional.
Let $\{h_1,\ldots, h_r\}$ be a Grobner basis of $I(S)$
with respect to a total degree ordering. Clearly, each $h_i$ vanishes on
\[
S = \{ x \in \re^n: \, -(f(x)-f_{min})=0, g_1(x) \geq 0, \ldots, g_m(x) \geq 0\}.
\]
By Real Nullstellensatz (cf. Corollary~4.4.3 of \cite{BCR}),
there exist $\ell \in \N$, $\varphi \in \re[x]$
and SOS polynomials $\phi_\nu$ ($\nu \in \{0,1\}^m$) such that
\[
h_i^{2\ell} - (f - f_{min}) \varphi +
\sum_{ \nu \in \{0,1\}^m } \phi_\nu g_{\nu} = 0.
\]
Applying $\mathscr{L}_{y^*}$ to the above
(suppose $2k$ is bigger than the degrees of all the above products)
results in
\be  \label{ReNul-Riesz}
\langle h_i^{2\ell}, y^* \rangle  +
\sum_{ \nu \in \{0,1\}^m } \langle \phi_\nu g_{\nu}, y^* \rangle =
\langle (f - f_{min}) \varphi, y^* \rangle.
\ee
By \reff{f-min+eps:sch}, for every $\eps>0$, we get
\be  \label{sum:vp*sig}
\langle (f - f_{min}+\eps) \varphi, y^* \rangle =
\sum_{ \nu \in \{0,1\}^m } \langle g_{\nu} \sig_\nu^\eps \varphi, y^* \rangle.
\ee

\medskip
\noindent
{\bf Step 2}\,
By \reff{Lgyp-->0} and item iv) of Lemma~\ref{mom-ker-limit}, we can get
\[
\lim_{ \eps \to 0} \,
\langle (g_{\nu} \sig_\nu^\eps \varphi, y^* \rangle = 0
\]
for $k$ big enough.
Hence, from \reff{sum:vp*sig} and the above, it holds that
\[
\langle (f - f_{min}) \varphi, y^* \rangle =  \lim_{ \eps \to 0} \,
\langle (f - f_{min}+\eps) \varphi, y^* \rangle =0.
\]
So, \reff{ReNul-Riesz} results in the equality
\[
\langle h_i^{2\ell}, y^* \rangle  +
\sum_{ \nu \in \{0,1\}^m } \langle \phi_\nu g_{\nu}, y^* \rangle = 0.
\]
Since $h_i^{2\ell}$ and every $\phi_\nu$ are SOS, each
$\langle \phi_\nu g_{\nu}, y^* \rangle \geq 0$
and $\langle h_i^{2\ell}, y^* \rangle \geq 0$,
by item ii) of Lemma~\ref{mom-ker-limit}.
Hence, $\langle h_i^{2\ell}, y^* \rangle = 0$,
i.e., $h_i^\ell \in \ker M_k(y^*)$. Again, by item i) of Lemma~\ref{mom-ker-limit}
and $M_k(y^*)=L_{g_0}^{(k)}(y^*)\succeq 0$, if $k$ is big enough, then
\[
h_i \in \ker M_k(y^*).
\]

\noindent
{\bf Step 3}\,
Like Step~3 in the proof of Theorem~\ref{thm:flat-put},
we would prove $y^*|_{2k-2}$ is flat by repeating the same argument there,
and omit it here for cleanness.
\qed

\section{Asymptotical Convergence}
\label{sec:cvg-las}
\setcounter{equation}{0}

In this section, we consider the case that
Lasserre's hierarchy of \reff{las-rex:put}
has asymptotic but not finite convergence.
Under the archimedean condition, Lasserre proved
$f_k \to f_{min}$ as $k\to\infty$.
Since it is possible that $f_k=f_k^*<f_{min}$ for every $k$,
we should not expect flat truncation holds in such a case.
When \reff{pop:K} has a unique global minimizer $u^*$,
Schweighofer \cite{Swg05} proved:
the subvector consisting of linear moments (indexed by $\af \in \N^n$ with $|\af|=1$)
of an almost optimizer $y^{(k)}$ of \reff{mom-las:put} for order $k$
converges to $u^*$ as $k\to\infty$.
When \reff{pop:K} has more than one minimizer,
do we have a similar convergence result?
To the author's best knowledge,
there is little work on this question.
This section is addressing this issue.
Generally, we need to use a higher order truncation of $y^{(k)}$,
and consider its limit points.
The main result of this section is that
every limit point of a truncation of $y^{(k)}$ is flat
if \reff{pop:K} has finitely many global minimizers.
In other words, flat truncation is asymptotically satisfied
when Lasserre's hierarchy has asymptotic convergence.

We assume the archimedean condition holds
for the tuple $g$: there exist $R>0$ and $k_0\in \N$ such that
\be \label{cond:AC}
R - \|x\|_2^2 \,\, \in \,\, Q_{k_0}(g).
\ee
For a Borel set $T \subseteq \re^n$, denote by ${\tt Prob}(T)$ the set of all
probability measures supported on $T$. Let $S$
be the set of global minimizers of \reff{pop:K}.
For each integer $t>0$, denote
\be \label{df:F2t}
F_{2t} := \left\{ \int_S [x]_{2t} d \mu:\, \mu \in {\tt Prob}(S) \right\}.
\ee

\begin{pro} \label{sol-tms-flat}
Let $S$ and $F_{2t}$ be defined as above. If $0<|S|<\infty$,
then for any integer $t\geq \max\{d_f,d_g + |S|-1\}$, every tms $z \in F_{2t}$ is flat.
\end{pro}
\begin{proof}
Let $\ell:=|S|$ and write $S=\{(a_{i,1},\ldots,a_{i,n}): i=1,\ldots,\ell\}$,
and $I$ be the ideal generated by the following polynomial equations
\[
(x_{i_1}-a_{1,i_1})(x_{i_2}-a_{2,i_2})\cdots(x_{i_\ell}-a_{l,i_n})=0,
\quad \forall \, i_1,\ldots, i_\ell \in \{1,\ldots,n\}.
\]
Clearly, the zero set of the above equations is $S$.
For each $\af \in \N^n$ with $|\af| = \ell$, the $x^\af$
is a leading monomial in one of the above defining polynomials,
so there exists $p_\af \in \re[x]_{\ell-1}$ such that
\be \label{p-af-in-I}
x^\af - p_\af  \equiv   0 \quad \mbox{ mod } \quad I.
\ee
Choose an arbitrary $z \in F_{2t}$, then there exists
$\mu \in {\tt Prob}(S)$ satisfying
\[
z = \int_S [x]_{2t} d \mu,
\qquad
M_t(z) = \int_S [x]_t [x]_t^T d \mu.
\]
By a simple induction on $|\af|$, one could show that for every $|\af| \in [\ell,t]$,
there exists $p_\af \in \re[x]_{\ell-1}$ satisfying \reff{p-af-in-I}.
Clearly, $x^\af-p_\af$ vanishes on $S$ and
\[
x^\af-p_\af \in \ker M_t(z).
\]
This means that every $\af$-th ($|\af|\geq \ell$) column of $M_t(z)$
is a linear combination of its $\bt$-th columns ($|\bt|\leq \ell-1$).
So, it holds that
\[
\rank \, M_{\ell-1}(z) = \rank \, M_r(z), \quad \mbox{ for } \,
r = \ell, \ell+1,\ldots,t.
\]
Thus, for all $t\geq d_g+\ell-1$, we have
\[
\rank \, M_{t-d_g}(z) = \rank \, M_t(z).
\]
Clearly, every $L_{g_i}^{(t)}(z) \succeq 0$,
and hence $z$ is flat by the definition.
\end{proof}

Denote by $\mathscr{M}_{\infty}$ the space of all full moment sequences
$w=(w_\af)$ indexed by vectors $\af\in \N^n$.
For all $w, z \in \mathscr{M}_{\infty}$, define
\[
\langle w, z \rangle = \sum_\af w_\af z_\af,
\quad \|w\|_2 = \sqrt{\langle w, w \rangle}.
\]
This induces a Hilbert space
\[
\mathscr{M}_{\infty}^2 = \{ w\in \mathscr{M}_{\infty}: \, \|w\|_2 <\infty\}.
\]
If we think of $\mathscr{M}_{\infty}^2$ as a Banach space,
then it is self-dual.
Clearly, \reff{pop:K} is equivalent to the optimization problem:
\be \label{K-mea-opt}
\min_{ y \in \mathscr{M}_{\infty} } \quad
\langle f, y \rangle \quad \mbox{s.t.} \quad y \in
\left\{\int_K [x]_{\infty} d \mu: \mu \in {\tt Prob}(K)\right \}.
\ee
The set of its optimizers is precisely ${\tt Prob}(S)$.

\begin{lem} \label{lm:||w||<=R}
Suppose \reff{cond:AC} holds.
If $y$ is feasible for \reff{mom-las:put} and $k\geq k_1 \geq k_0$, then
\be \label{w:k-k0<=bd}
\| y|_{2(k_1-k_0)} \|_2^2 \leq  1+R+\cdots + R^{k_1-k_0} .
\ee
\end{lem}

\begin{proof}
By \reff{cond:AC}, there exist
SOS polynomials $s_0, s_1, \ldots, s_m$ such that
\[
R - \|x\|_2^2 =  \sum_{i=0}^m g_is_i, \, \mbox{ each } \, \deg(g_is_i) \leq 2k_0.
\]
(Here $\|x\|_2:=\sqrt{x_1^2+\cdots+x_n^2}$.)
Thus, for every $ j = 1, \ldots, k_1-k_0$, we have
\[
R \cdot \|x\|_2^{2j-2} - \|x\|_2^{2j} =
\sum_{i=0}^m g_i \cdot s_i \cdot \|x\|_2^{2j-2}.
\]
By item ii) of Lemma~\ref{mom-ker-limit},
applying $\mathscr{L}_y$ to the above gives
\[
R \mathscr{L}_y(\|x\|_2^{2j-2}) - \mathscr{L}_y(\|x\|_2^{2j}) \geq 0,
\quad j = 1, \ldots, k_1-k_0.
\]
Therefore, it holds that
\[
\mathscr{L}_y(\|x\|_2^2) \leq R, \quad \mathscr{L}_y(\|x\|_2^4) \leq R^2,
\ldots, \quad \mathscr{L}_y(\|x\|_2^{2j}) \leq R^j.
\]
The moment matrix $M_k(y) = L_{g_0}^{(k)}(y)\succeq 0$ implies
its submatrix $M_{k_1-k_0}(y)\succeq 0$. In the below,
we denote by $\|A\|_F$ the Frobenius norm of a matrix $A$, i.e.,
$\|A\|_F = \sqrt{Trace(A^TA)}$.
Recall that $\|A\|_F \leq Trace(A)$ if $A \succeq 0$. Clearly,
\[
\sum_{ |\af| \leq 2(k_1-k_0) } |y_\af|^2 \leq \| M_{k_1-k_0}(y) \|_F^2
\leq (Trace(M_{k_1-k_0}(y)))^2,
\]
\[
Trace(M_{k_1-k_0}(y)) = \sum_{j=0}^{k_1-k_0} \sum_{|\af|=2j} y_\af
\leq \sum_{j=0}^{k_1-k_0}  \mathscr{L}_y(\|x\|_2^{2j})
\leq \sum_{j=0}^{k_1-k_0} R^j.
\]
So, the inequality \reff{w:k-k0<=bd} is true.
\end{proof}

It is possible that \reff{mom-las:put} might not have a minimizer,
i.e., the optimal value $f_k^*$ of \reff{mom-las:put}
may not be achievable (cf. \cite[Example~4.8]{Swg05}). Thus, we consider
an almost optimizer of \reff{mom-las:put}.
Let $\{y^{(k)}\}$ be a sequence such that
$y^{(k)}$ is feasible for \reff{mom-las:put} with order $k$.
We say $\{y^{(k)}\}$ is {\it asymptotically optimal} if
\[
\lim_{k\to\infty} \langle f, y^{(k)} \rangle =
\lim_{k\to\infty} f_k^*.
\]
(A different notion {\it nearly optimality} was used in \cite{Swg05}.)
Note that if \reff{cond:AC} holds and $\{y^{(k)}\}$ is asymptotically optimal,
then $\langle f, y^{(k)} \rangle \to f_{min}$ as $k\to\infty$.

\begin{theorem} \label{thm:y->flatrun}
Assume the archimedean condition \reff{cond:AC} holds and
the set $S$ of global minimizers of \reff{pop:K} is nonempty and finite.
Let $\{y^{(k)}\}$ be asymptotically optimal for \reff{mom-las:put}.
Then, for every  $t\geq \max\{d_f,d_g +|S|-1\}$,
the truncated sequence  $\{y^{(k)}|_{2t}\}$ is bounded,
and its every limit point belongs to $F_{2t}$ and is flat.
\end{theorem}

\begin{proof}
If we replace $k_1-k_0$ by $t$ in \reff{w:k-k0<=bd} of Lemma~\ref{lm:||w||<=R},
then the sequence  $\{y^{(k)}|_{2t}\}$ is clearly bounded.
One could generally assume $R<1$, because otherwise we can scale
$x$ in \reff{cond:AC} so that $R<1$ and Lasserre's hierarchy remains equivalent.

Let $v$ be an arbitrary limit point of $\{y^{(k)}|_{2t}\}$.
One could generally assume $y^{(k)}|_{2t} \to v$.
We need to prove $v \in F_{2t}$.
Suppose otherwise $v \not\in F_{2t}$. For every $k$, define
\[
z^{(k)} := y^{(k)}|_{2(k-k_0)}.
\]
Each $z^{(k)}$ could be treated as a vector in
$\mathscr{M}_{\infty}^2$ by adding zero entries to the tailing.
So, for $k > k_0$, the inequality \reff{w:k-k0<=bd} implies
\[
\| z^{(k)} \|_2^2 \leq 1 + R + R^2 + R^3 + \cdots  = 1/(1-R).
\]
The sequence $\{ z^{(k)} \}$ is bounded in $\mathscr{M}_{\infty}^2$.
By Alaoglu's Theorem (cf. \cite[Theorem~V.3.1]{Conway} or \cite[Theorem~C.18]{LasBok}),
it has a subsequence $\{z^{(k_j)} \}$ that is convergent in the weak-$\ast$ topology.
That is, there exists $z^* \in \mathscr{M}_{\infty}^2$ such that
\[
\langle c, z^{(k_j)} \rangle \, \to \,  \langle c,  z^* \rangle
 \quad \mbox{ as } j \to \infty
\]
for every $c \in \mathscr{M}_{\infty}^2$.
If we choose $c$ as $\langle c, w \rangle= w_\af$ for each $\af$,
then one could get
\be \label{2k2t->w*2t}
 y^{(k_j)}|_{\af} = z^{(k_j)}|_{\af}\to  z^*|_{\af}, \qquad z^*|_{2t} = v.
\ee
Note that every $L_{g_i}^{(r)}(z^{(k_j)})\succeq 0$ if $k_j \geq 2r$.
By \reff{2k2t->w*2t}, it holds that for all $r$
\[
L_{g_i}^{(r)}(z^*) \succeq 0, \quad i=0,1,\ldots,m.
\]
Hence, $z^* \in \mathscr{M}_{\infty}^2$ is a full moment sequence
such that the corresponding localizing matrices of
all orders are positive semidefinite.
By Lemma~3.2 of Putinar \cite{Put} and \reff{cond:AC}, $z^*$ admits a probability measure
(note $\mathscr{L}_{z^*}(1)=1$) supported on $K$, and it is feasible for \reff{K-mea-opt}.
The polynomial $f$ defines a continuous linear functional acting on
$\mathscr{M}_{\infty}^2$ as $\langle f, w \rangle$.
By the weak-$\ast$ convergence of $z^{(k_j)}\to z^*$, it holds that
\[
\langle f, z^* \rangle = \lim_{j \to \infty} \langle f, z^{(k_j)} \rangle
= \lim_{j \to \infty} \langle f, y^{(k_j)} \rangle  =
\lim_{j \to \infty} f_{k_j}^*  = f_{min}.
\]
The above last two equalities are because \reff{cond:AC} holds
and the sequence $\{y^{(k)}\}$ is asymptotically optimal.
This means that $z^*$ is also a minimizer of \reff{K-mea-opt}.
Hence, $v=z^*|_{2t} \in F_{2t}$.
But, this contradicts the earlier assertion that $v\not\in F_{2t}$.
Therefore, every limit point of $\{y^{(k)}|_{2t}\}$
belongs to $F_{2t}$. By Proposition~\ref{sol-tms-flat},
we know every $v\in F_{2t}$ is flat if $t \geq \max\{d_f, d_g+|S|-1\}$.
\end{proof}

%
%
%
%

In Theorem~\ref{thm:y->flatrun}, for a fixed $t$, it is possible that
the truncation $y^{(k)}|_{2t}$ is not flat for every $k$.
But, for every $\eps>0$, if $k$ is big enough, there exists a $z\in F_{2t}$
such that $\|y^{(k)}|_{2t} - z \|_2 \leq \eps.$
In other words, flat truncation is asymptotically satisfied
when $S$ is finite.
Theorem~\ref{thm:y->flatrun} also implies that
the distance between $y^{(k)}|_{2t}$ and $F_{2t}$
tends to zero as $k \to \infty$.
Therefore, in numerical computations, flat truncation has
a good chance to be satisfied,
even if Lasserre's hierarchy does not have finite convergence.

\begin{remark}
A similar version of Theorem~\ref{thm:y->flatrun} also holds
for Schm\"{u}dgen type Lasserre's hierarchy.
This is because \reff{mom-sdp:schmg} can be thought of as a Putinar type one
applied to the $2^m$ constraints $g_\nu(x) \geq 0 \,(\nu \in \{0,1\}^m)$.
The archimedean condition \reff{cond:AC} for the corresponding constraints
is automatically satisfied for compact $K$, by Schm\"{u}dgen's Positivstellensatz.
Hence, when $K$ is compact and $S$ is finite, flat truncation is also asymptotically satisfied
for the hierarchy of \reff{mom-sdp:schmg}.
\end{remark}

\section{Some Applications}
\label{sec:spec-cas}
\setcounter{equation}{0}

In this section, we consider two interesting cases
of semidefinite relaxations for solving
polynomial optimization, for which
flat truncation could serve as a certificate
to check their exactness and be used to get minimizers.

\subsection{Standard SOS relaxations}

For a polynomial $f\in \re[x]$,
consider the unconstrained polynomial optimization problem
\be  \label{min-f:R^n}
\underset{x\in \re^n}{\min} \qquad  f(x).
\ee
To have a finite minimum $f_{min}$,
assume $f$ has an even degree $2d$.
The standard SOS relaxation (cf. \cite{Las01,PS03})
for solving \reff{min-f:R^n} is
\be \label{sos-rex:R^n}
 \max \quad \gamma \quad
\mbox{ s.t. }  \quad f - \gamma  \quad \mbox{ is SOS}.
\ee
Its dual optimization problem is
\be \label{mom-sdp:R^n}
 \min_{ y \in \mathscr{M}_{2d} } \quad \langle f, y \rangle \quad
\mbox{ s.t. }  M_d(y) \succeq 0, \, \langle 1, y \rangle = 1.
\ee
Lasserre \cite{Las01} showed that \reff{sos-rex:R^n} and \reff{mom-sdp:R^n}
have the same optimal value (there is no duality gap), which we denote by $f_{sos}$,
and \reff{sos-rex:R^n} achieves its optimum if $f_{sos}>-\infty$,
because \reff{mom-sdp:R^n} has an interior point.
Clearly, $f_{sos} \leq f_{min}$.
As demonstrated by the numerical experiments of \cite{PS03},
it occurs quite a lot that $f_{min}=f_{sos}$.
Thus, one is wondering:
how do we check $f_{sos} = f_{min}$,
and if so how do we get minimizers of \reff{min-f:R^n}?
This issue could be solved by using flat truncation.
%
%

The whole space $\re^n$ would be defined by the trivial inequality $1 \geq 0$.
Thus, the associated $k$-th truncated quadratic module and preordering
coincide and
\[
Q_k(1) = Pr_k(1) = \Sig_{n,2k},
\]
where $\Sig_{n,2k}$ denotes the cone of SOS polynomials
having $n$ variables and degree $2k$.
Like \reff{las-rex:put} and \reff{mom-las:put}, the $k$-th order Lasserre's relaxation
for \reff{min-f:R^n} is
\be \label{kth-sos:R^n}
 \max \quad \gamma \quad
\mbox{ s.t. }  f - \gamma  \in \Sig_{n,2k},
\ee
and its dual optimization problem is
\be \label{kth-mom:R^n}
 \min_{ y \in \mathscr{M}_{2k} }  \quad \langle f, y \rangle \quad
\mbox{ s.t. }  M_k(y) \succeq 0, \, \langle 1, y \rangle = 1.
\ee
Clearly, for every $k \geq d$, \reff{kth-sos:R^n} is equivalent to \reff{sos-rex:R^n}.
However, \reff{kth-mom:R^n} and \reff{mom-sdp:R^n} are different
in satisfying flat truncation, though they have the same optimal value.
Theorem~\ref{flatrun:exc-rlx} implies the following.

\begin{cor}
Let $f_{sos}$ be the optimal value of \reff{sos-rex:R^n}.
Suppose $f_{min}=f_{sos}$ and \reff{min-f:R^n} has
a nonempty set of finitely many global minimizers.
Then, for $k$ big enough, every minimizer $y^*$ of \reff{kth-mom:R^n}
has a flat truncation.
\end{cor}

If $f_{min} > -\infty$, then $f$ generically has finitely many
minimizers (cf. \cite[Prop.~2.1]{NR09}). Note that
for the case $f_{sos} =f_{min}$,
if $y^*$ is an optimizer of \reff{mom-sdp:R^n}
instead of \reff{kth-mom:R^n},
then $y^*$ might not have a flat truncation.
In this sense, \reff{kth-mom:R^n} is stronger than \reff{mom-sdp:R^n},
though their optimal values are same.
We show the difference between \reff{kth-mom:R^n} and \reff{mom-sdp:R^n}
in satisfying flat truncation as follows.

\begin{exm}
Consider the polynomial optimization problem
\[
\min_{x\in \re^3} \quad (x_1x_2-1)^2+(x_1x_3-1)^2+(x_2x_3-1)^2.
\]
Let $f$ be the objective. Clearly, $f_{sos}=f_{min}=0$.
The global minimizers are $\pm (1,1,1)$.
However, not every minimizer of the corresponding
\reff{mom-sdp:R^n} has a flat truncation.
For instance, for all $a>0$, the tms
\footnote{Its entries are listed in graded lexicographical ordering.}
\[
%
%
\baray{c}
y^*(a):=
(1,\ 0,\  0,\  0,\  1,\   1,\   1,\    1,\   1,\  1,\ 0,\  0,\   0,\  0,\ 0,\ 0 ,\ 0,\  0,\  0 ,\ 0, \\
   1+a,\   1,\    1,\    1,\    1,\   1,\    1,\    1,\   1,\   1,\
   1+a,\    1,\    1,\    1,\   1+a)
\earay
\]
is a minimizer of \reff{mom-sdp:R^n}.
For all $a>0$, one can verify that
\[
\rank\, M_0(y^*(a))=1, \quad \rank\, M_1(y^*(a))=2, \quad \rank\, M_2(y^*(a))=5.
\]
Thus, $y^*(a)$ does not have a flat truncation if $a>0$.
If we solve \reff{mom-sdp:R^n} by {\tt SeDuMi},
a numerical value $2.7232$ of $a$ is returned,
and the computed minimizer does not have a flat truncation.
For $k=3$, the numerical optimizer of \reff{kth-mom:R^n}
returned by {\tt SeDuMi} does not have a flat truncation.
However, for $k=4$, the numerical optimizer of \reff{kth-mom:R^n}
returned by {\tt SeDuMi} has a flat truncation,
and we are able to get the two global minimizers.
\qed
\end{exm}

\subsection{Jacobian SDP relaxation}

Now we consider Jacobian SDP relaxations for polynomial optimization
introduced by the author in the earlier work \cite{Nie:JacSDP}.
It is a refining of relaxation \reff{sos-rex:shmg}.
Its basic idea is to introduce some redundant polynomial equalities, say,
\[
\varphi_1(x) = \cdots = \varphi_L(x) = 0,
\]
which are constructed from the minors of Jacobians of $f, g_1, \ldots, g_m$.
For instance, when \reff{pop:K} has a single inequality constraint, say $g(x) \geq 0$,
the newly introduced equalities are
\[
g(x)\cdot f_{x_i}(x)= 0, \, i=1,\ldots,n,
\]
\[
\sum_{1\leq i<j\leq n, i+j=\ell} \Big(
f_{x_i}(x) g_{x_j}(x)  - f_{x_j}(x) g_{x_i}(x) \Big) =0, \,
\ell = 3,\ldots, 2n-1.
\]
When \reff{pop:K} achieves its minimum and some generic nonsingularity
conditions (see Assumption~2.2 of \cite{Nie:JacSDP}) holds,
\reff{pop:K} is equivalent to
\be \label{pop:Jac-K}
\left\{\baray{rl}
\underset{x\in \re^n}{\min} &  f(x)\\
\mbox{s.t.} & \varphi_1(x) = \cdots = \varphi_L(x) = 0, \\
& g_1(x)\geq 0, \ldots, g_m(x) \geq 0.
\earay \right.
\ee
Let $PJ_k$ denote the $k$-th truncated preordering
for \reff{pop:Jac-K} (cf. \cite[Sec.~2.2]{Nie:JacSDP}).
Then, the resulting version of relaxation \reff{sos-rex:shmg}
for \reff{pop:Jac-K} is
\be  \label{sos:jac}
\max \quad \gamma \quad
\mbox{ s.t. } \quad f - \gamma  \in PJ_k.
\ee
Similarly, its dual optimization problem is (cf. \cite[Sec.~2.2]{Nie:JacSDP})
\be  \label{mom:jac}
\left\{ \baray{rl}
\underset{ y \in \mathscr{M}_{2k} }{\min}  & \langle f, y \rangle \, \\
\mbox{ s.t. } & L_{\varphi_j}^{(k)}(y)=0, j=1,\ldots,L, \\
&  L_{g_{\nu}}^{(k)}(y)\succeq 0 \,(\forall \nu\in\{0,1\}^m),
\, \langle 1, y \rangle = 1.
\earay \right.
\ee

An attractive property of the hierarchy of \reff{sos:jac}
is that it always has finite convergence.
Thus, a practical concern in applications is:
how does one identify its finite convergence?
Interestingly, flat truncation could serve as a certificate for this purpose
when \reff{pop:K} has finitely many minimizers.

\begin{cor}  \label{cor:jac-sdp}
Suppose \reff{pop:K} has a nonempty set of finitely many global minimizers
and its feasible set is nonsingular
(Assumption~2.2 of \cite{Nie:JacSDP} is satisfied).
Then, for all $k$ big enough, the optimal value of \reff{sos:jac}
equals the global minimum of \reff{pop:K}
and every minimizer of \reff{mom:jac} has a flat truncation.
\end{cor}
\begin{proof}
As before, denote by $\widetilde{f}_k$ the optimal value of \reff{sos:jac}.
By Theorem~2.3 of \cite{Nie:JacSDP}, there exists $k_1$
such that $\widetilde{f}_{k_1}$ is equal to the global minimum of \reff{pop:Jac-K},
which is also equal to $f_{min}$,
when Assumption~2.2 of \cite{Nie:JacSDP} is satisfied.
So, there is no duality gap between \reff{sos:jac}
and \reff{mom:jac} when $k$ is big.
Clearly, \reff{pop:Jac-K} also has finitely many global minimizers,
since it is equivalent to \reff{pop:K}.
Thus, the conclusion of this corollary just follows from Theorem~\ref{flatrun:exc-rlx}.
\end{proof}

In particular, if a polynomial has finitely many global minimizers
in the whole space $\re^n$, then its minimizers
could be obtained by using gradient SOS relaxation \cite{NDS},
which is a special case of Jacobian SDP relaxation
(see Corollary~2.6 of \cite{Nie:JacSDP}).
Its dual version of \reff{mom:jac} becomes
\be  \label{mom:grad}
\left\{ \baray{rl}
\underset{ y \in \mathscr{M}_{2k} }{\min}   & \langle f, y \rangle \, \\
\mbox{ s.t. } & L_{f_{x_j}}^{(k)}(y)=0, j=1,\ldots,n, \\
&  M_k(y)\succeq 0, \, \, \, \langle 1, y \rangle = 1.
\earay \right.
\ee
Hence, we could also get the following.

\begin{cor}
Suppose a polynomial $f$ has finitely many global minimizers in $\re^n$.
Then, for all $k$ big enough, the optimal value of \reff{mom:grad}
equals the global minimum of $f$ in $\re^n$,
and its every minimizer has a flat truncation.
\end{cor}

\section{Some discussions}

Our main result is that, in the generic case (e.g., Assumption~\ref{crt-pt<inf} holds),
Putinar type Lasserre's hierarchy has finite convergence
if and only if the flat truncation condition \reff{cd:flat-trun}
is satisfied at every minimizer of
the dual optimization problem \reff{mom-las:put} for $k$ big enough.
Under the archimedean condition, we also showed that
\reff{cd:flat-trun} is always asymptotically satisfied
if there are finitely many global minimizers.
(We have similar conclusions for Schm\"{u}dgen type Lasserre's hierarchy.)
This result has applications
in solving polynomial optimization problems of the form \reff{pop:K}.
One could solve \reff{mom-las:put}, starting with an order $k\geq \max\{d_f,d_g\}$.
If \reff{cd:flat-trun} is satisfied,
then we stop and can get one or several global minimizers;
if \reff{cd:flat-trun} is not satisfied,
then increase $k$ by one and solve \reff{mom-las:put} again.
By repeating this procedure, we solve \reff{mom-las:put}
and then check \reff{cd:flat-trun} iteratively. In the generic case,
when Lasserre's hierarchy has finite convergence,
our result implies that this procedure must terminate
within finitely many steps;
when Lasserre's hierarchy has only asymptotic convergence,
our result shows that this procedure will terminate asymptotically.
Flat truncation has been used algorithmically
in the solution extraction mechanism
by the polynomial optimization package
{\tt GloptiPoly~3} \cite{GloPol3}.

We would like to make some further remarks about the difference
between the flat truncation condition \reff{cd:flat-trun}
and the flat extension condition \reff{con:fec}.
Consider the case that Lasserre's hierarchy has finite convergence,
say, $f_k^* = f_k =f_{min}$ for all $k\geq k_0$.
Assume $k_0 \geq d:=\max\{d_f, d_g\}$.
Suppose $y^*$ is a minimizer of \reff{mom-las:put} with order $k > k_0$
and \reff{cd:flat-trun} is satisfied for $t\in [k_0+d_g, k]$.
Clearly, the truncation $\hat{y}:=y^*|_{2k_0}$ is also
a minimizer of \reff{mom-las:put} with order $k_0$,
and $\hat{y}$ has a flat extension $y^*|_{2t}$.
One is thus interested in the question whether
every minimizer of \reff{mom-las:put} with order $k_0$
is a truncation of a minimizer of \reff{mom-las:put} with order $k>k_0$
that satisfies \reff{cd:flat-trun}?
Or equivalently, does every minimizer of \reff{mom-las:put} with order $k_0$
has a flat extension? Unfortunately, this may not be true.
We show this by using Example~\ref{emp:FTneFE}.
By {\tt SeDuMi}, the computed optimizer $z^*$ of \reff{mom-las:put} with order $2$
is $z^*:=(1.0000,0.0000,0.0000,0.0000,0.7908)$, and
\[
M_2(z^*) =
\left[\baray{lll}
1.0000 & 0.0000  & 0.0000  \\ 0.0000  & 0.0000  & 0.0000  \\
0.0000  & 0.0000  &  0.7908
\earay \right].
\]
The tms $z^*$ is not {\it recursively generated}
(cf. \cite[Section~2]{CF05}), because the coefficient vector of the polynomial $x-1$
belongs to the kernel of $M_2(z^*)$ but
the coefficient vector of the multiple $x(x-1)$ does not.
So, $z^*$ does not admit any representing measure (cf. \cite[Corollary~2.11]{CF05}),
and hence $z^*$ has no flat extension.

It is interesting to know for how big $k$ every minimizer of
the dual optimization problem \reff{mom-las:put}
satisfies the flat truncation condition \reff{cd:flat-trun}.
The existence of such a $k$ is guaranteed
by Theorem~\ref{thm:flat-put} for the generic case,
but there is no estimate on it.
Clearly, any bound on such $k$ depends on $f$.
Thus, one is interested to know whether there is a uniform bound $N$,
which only depends on the defining polynomials $g_i$ of the set $K$ and the degree of $f$,
such that every minimizer of \reff{mom-las:put} with order $N$
satisfies \reff{cd:flat-trun} for all generic $f$ with the given degree?
If such a uniform bound $N$ would exist, it would be very useful in applications.
This is because one would only need to solve \reff{mom-las:put} for a fixed order $N$;
if \reff{cd:flat-trun} is satisfied, we know Lasserre's hierarchy has finite convergence;
if not, then Lasserre's hierarchy has no finite convergence in the generic case.
However, such a uniform bound $N$ typically does not exist.
For instance, there is no such an $N$ when $K$ is the three dimensional unit ball and
$\deg(f)=6$. Consider the Motzkin polynomial
$M:=x_1^4x_2^2+x_1^2x_2^4+x_3^6-3x_1^2x_2^2x_3^2$.
There exists a sequence of generic polynomials $p_k$ of degree six
such that $p_k \to M$ as $k \to \infty$.
Let $\gamma_k$ be the minimum value of $p_k$ on the unit ball.
Clearly, $\gamma_k \to 0$. The unit ball has nonempty interior,
so \reff{las-rex:put} achieves its optimal value for every $k$ (cf. \cite{Las01}).
If, otherwise, such $N$ would exist, then $p_k-\gamma_k \in Q_N(1-\|x\|_2^2)$ for all $k$.
Since the cone $Q_N(1-\|x\|_2^2)$ is closed (cf. \cite[Theorem~3.33]{Lau}),
the Motzkin polynomial $M$, which is the limit of $p_k-\gamma_k$ as $k \to \infty$,
must also belong to $Q_N(1-\|x\|_2^2)$.
However, this is not true (cf. \cite[Example~5.3]{Nie:JacSDP}).
Therefore, such a uniform bound $N$ does not exist.
Any bound on $k$, for which
every minimizer of \reff{mom-las:put} satisfies \reff{cd:flat-trun},
must depend on $f$, even if $f$ has generic coefficients.
The dependence relation between such a bound on $k$ and a generic polynomial $f$
is an interesting future research question.
%
%
%

\bigskip \noindent
{\bf Acknowledgement}\,
The author was partially supported by NSF grants
DMS-0757212 and DMS-0844775,
and he would like very much to thank the referees
for fruitful suggestions on improving the paper.

\end{document}